\documentclass[a4paper]{article}
\usepackage[utf8]{inputenc}
\usepackage[T1]{fontenc}

\usepackage{amsmath,amsfonts,amssymb}
\usepackage{mathtools}
\usepackage{amsthm}
\usepackage[]{authblk} 
\usepackage[svgnames]{xcolor}
\usepackage[]{geometry}

\usepackage[]{hyperref}
\hypersetup{colorlinks=true,
  citecolor=DarkBlue,
  linkcolor=DarkBlue,
  linktoc=page,
  urlcolor=black,
}

\title{A family of Beckner inequalities under various curvature-dimension conditions}
\date{}

\author{Ivan Gentil\thanks{\texttt{\href{mailto:gentil@math.univ-lyon1.fr}{gentil@math.univ-lyon1.fr}}}}
\author{Simon Zugmeyer\thanks{\texttt{\href{mailto:zugmeyer@math.univ-lyon1.fr}{zugmeyer@math.univ-lyon1.fr}}}}

\affil{Univ Lyon, Université Claude Bernard Lyon 1, CNRS UMR 5208, Institut Camille Jordan, 43 blvd. du 11 novembre 1918, F-69622 Villeurbanne cedex, France.}

\newtheorem{thrm}{Theorem}[section]
\newtheorem{prop}[thrm]{Proposition}
\newtheorem{coro}[thrm]{Corollary}
\newtheorem{lemm}[thrm]{Lemma}

\theoremstyle{definition}
\newtheorem{defi}[thrm]{Definition}

\theoremstyle{remark}
\newtheorem{rmrk}[thrm]{Remark}

\newcommand{\R}{\mathbb{R}}

\newcommand{\vphi}{\varphi}
\newcommand{\eps}{\varepsilon}

\newcommand{\ol}{\overline}

\DeclareMathOperator{\Var}{Var}
\DeclareMathOperator{\Hess}{Hess}
\DeclareMathOperator{\Ent}{Ent}
\DeclareMathOperator{\Ric}{Ric}
\DeclareMathOperator{\Id}{Id}

\DeclarePairedDelimiter{\norm}{\lVert}{\rVert} 
\DeclarePairedDelimiter{\abs}{\lvert}{\rvert}
\DeclarePairedDelimiter{\sbra}{[}{]}
\DeclarePairedDelimiter{\pare}{(}{)}

\makeatletter
\renewcommand\d[1]{\mspace{2mu}\mathrm{d}#1\@ifnextchar\d{\mspace{-1mu}}{}}
\let\oldabs\abs
\def\abs{\@ifstar{\oldabs}{\oldabs*}}
\let\oldnorm\norm
\def\norm{\@ifstar{\oldnorm}{\oldnorm*}}
\let\oldsbra\sbra
\def\sbra{\@ifstar{\oldsbra}{\oldsbra*}}
\let\oldpare\pare
\def\pare{\@ifstar{\oldpare}{\oldpare*}}
\makeatother

\begin{document}

\maketitle

\begin{abstract}
  In this paper, we offer a proof for a family of functional inequalities interpolating between the Poincaré and the logarithmic Sobolev (standard and weighted) inequalities. The proofs rely both on entropy flows and on a \(CD(\rho,n)\) condition, either with \(\rho=0\) and \(n>0\), or with \(\rho>0\) and \(n\in\R\). As such, results are valid in the case of a Riemannian manifold, which constitutes a generalization to what was proved in~\cite{bgs,18nguyen}.
  \medskip
  
  \noindent\textbf{Keywords:}  Curvature-dimension condition, Poincar\'e inequality, Beckner inequalities, Entropy flows.
\end{abstract}


\section{Introduction}
\label{sec-1}

The family of the Beckner inequalities interpolate between the Poincar\'e and the logarithmic Sobolev inequalities. For instance, let $d\mu=e^{-V}dx$ be a probability measure on $\R^d$, where $V:\R^d\to \R$ is a smooth function satisfying $\nabla^2 V\geq \rho\Id$ for some $\rho>0$. Then, Bakry-\'Emery's curvature-dimension condition implies that, for any $p\in(1,2]$, and any nonnegative smooth function $f$, 
\begin{equation}
  \label{eq-1}
  \frac{p}{p-1}\sbra{\int f^2d\mu-\left(\int f^{2/p}d\mu\right)^{\!\!p}}\leq
  \frac{2}{\rho}
  \int \abs{\nabla f}^2d\mu,
\end{equation}
these results can be found in~\cite[Sec.~7.6.2]{bgl-book}.
When $p=2$, that is the usual  Poincar\'e inequality for the measure $\mu$  and when $p\rightarrow 1$ the inequality becomes the logarithmic Sobolev inequality,
$$
\int f^2\log\frac{f^2}{\int f^2d\mu}d\mu\leq \frac{2}{\rho}
\int |\nabla f|^2d\mu,
$$
both optimal when $\mu$ is the standard Gaussian measure. 

Similar inequalities have first been proved by Bidaut-V\'eron-V\'eron in~\cite{bidaut-veron} for the sphere, using the method proposed in~\cite{gidas-spruck}. In this article, we refer to them as Beckner inequalities, in reference to~\cite{beckner89}, where Beckner proves inequallity~\eqref{eq-1} for the Gaussian measure. They are sometimes called convex inequalities. 

These inequalities play a major role among functional inequalities for probability measures, being useful to understand both asymptotic behaviour of parabolic equations, and also the geometry of measured spaces.  It is interesting to notice that for the Gaussian measure, although the Poincar\'e and the logarithmic Sobolev inequalities are optimal and have nonconstant extremal functions, whenever $p\in\pare{1,2}$, Beckner inequalities do not have extremal functions, and even admit various improvements (see for instance~\cite{arnolddolbeault05,bolley-gentil,dekl2}).

Other attractive examples appear for measures which satisfy a (weighted) Poincar\'e inequality but no logarithmic Sobolev inequality. More precisely, let $\varphi:\R^d\mapsto\R_+^*$ be a smooth and positive function such that for any $\beta>d/2$, $\int \varphi^{-\beta}dx<+\infty$. Now, define
\begin{equation}
  d\mu_\beta={c_\beta}{\vphi^{-\beta}}dx,
\end{equation}
where $c_\beta$ is a normalization constant such that $\mu_\beta$ is a probability measure on $\R^d$.
In~\cite{bgs}, the authors prove that, when $\varphi=1+|x|^2$ then for any $\beta\geq d+1$ and $p\in[p^*,2]$.  
\begin{equation}
  \label{eq-2}
  \frac{p}{p-1}\left[\int f^2d\mu_\beta-\left(\int f^{2/p}d\mu_\beta\right)^{\!\!p}\right]\leq 
  \frac{1}{\beta-1}
  \int |\nabla f|^2\varphi d\mu_\beta,
\end{equation}
where $p^*=1+1/(\beta-d)$. This inequality is rich enough to be equivalent to the Sobolev inequality on the sphere.  This result has then been extended by N'Guyen in~\cite{18nguyen}, where the author proves similar inequalities in $\R^d$ with a function $\varphi$ satisfying the convex assumption $\nabla^2 \varphi\geq c\Id$ for some constant \(c>0\). The only difference with inequality \eqref{eq-2} is that the constant $(\beta-1)^{-1}$, in front of the right hand side, becomes $2(c(\beta-1))^{-1}$, which is consistant with previous results, since \(c=2\) for \(\vphi:x\mapsto1+\abs{x}^2\). In that regard, N'Guyen's result is more general, but there exists a limitation: the range of the parameter \(p\) for which the inequality remains valid is strictly smaller.

\medskip

We would like to extend Beckner's inequality 
more precisely inequality \eqref{eq-2} in the general context of curvature-dimension conditions. The goal is twofold. First, we extend some results of~\cite{bgs,18nguyen} in the context of a Riemannian manifold. Then, under a curvature-dimension condition $CD(\rho,n)$ with $n$ \emph{negative}, we also prove Beckner inequalities. We recover for instance the weighted Poincar\'e  inequality  for generalized Cauchy distributions.

More precisely, we prove functional inequalities under two kinds of assumptions,  the curvature-dimension conditions $CD(0,n)$ with $n>0$ or $CD(\rho,n)$ with $\rho>0$ and $n\in\R$. Let us give here a flavour of our results.
\begin{itemize}
\item  In Theorem~\ref{th:weighted_poincare}, we prove the following Poincar\'e inequality under $CD(0,n)$, $n>0$. Let $(M,g)$ be a smooth $d$-dimensional Riemannian manifold with a nonnegative Ricci curvature and let $\vphi$ be a positive function such that $\Hess(\vphi)\geq c g$, $c>0$. Then for any function $f$ and $\beta\geq d+1$, 
\begin{equation}
  \label{eq-19}
  \Var_{\mu_{\vphi,\beta}}(f)\leq\frac{1}{c(\beta-1)}\int \Gamma(f)\vphi d\mu_{\vphi,\beta},
\end{equation}
where $\mu_{\vphi,\beta}=Z_{\vphi,\beta}\vphi^{-\beta}dx$ is a probability measure. This inequality generalizes previous results on the subject. 
\item  In Theorem~\ref{th:beckner}, we prove a family of Beckner inequalities under $CD(\rho,n)$ with $\rho>0$ and $n\in\R$.  On a $d$-dimensional Riemannian manifold, for any nonnegative function $f$, 
$$    
\frac{p}{p-1}\pare{\int f^2d\mu-\pare{\int f^{2/p}d\mu}^{\!\!p} } \leq 2\frac{n-1}{\rho n}\int\Gamma (f)d\mu.
$$
\begin{itemize}
    \item for all \(p\in \left(1,2\right]\) if \(n\geq d\),
    \item for all \(p\in \left[p^*,2\right]\) if \(n< -2\), where $p^*=1+\frac{1-4n}{2n^2+1}$.
 \end{itemize}
In this context, $\mu$ is the reversible measure and $\Gamma$ is the carr\'e du champ operator. 
The use of a negative dimension is new, up to our knowledge. 
\end{itemize}
\medskip

The article is organized as follows. In the next section, we state various definitions useful for the rest of the paper. In Section~\ref{sec-3} we prove weighted Beckner inequalities, like inequality~\eqref{eq-19} under $CD(0,n)$ conditions, $n>0$. The $\Phi$-entropy inequalities are also studied in this context. Section~\ref{se:CD(rho,n)} is devoted to Beckner's inequality under the curvature-dimension condition $CD(\rho,n)$ conditions with $\rho>0$ and $n\in\R$. Finally in Section~\ref{sec-5}, we apply our methods in the one dimensional case, giving another way to prove optimal weighted Poincar\'e inequality on $\R$. 

\bigskip

\noindent{\bf Acknowledgements:}
This research was supported  by the French ANR-17-CE40-0030 EFI project. The authors warmly thank L. Dupaigne for discussing the problem and proofreading a draft version of this work.


\section{Settings and definitions}\label{sec-2}

Consider a connected, \(C^\infty\) Riemannian manifold \((M,g)\) of dimension \(d\) and the Laplace-Beltrami operator \(\Delta_g\) given by, in a local chart
\[
\Delta_gf=\frac{1}{\sqrt{\abs g}}\partial_i\pare{\sqrt{\abs g}g^{ij}\partial_j f},
\]
where \(g^{ij}\) are the components of the inverse metric tensor, \(g^{-1}\). In this formula and in what follows, the Einstein notation was used, where the summation on indices is implied. 

On this manifold, define the symmetric diffusion operator \(L=\Delta_g+\Gamma^{\Delta_g}(V,.)\), where \(\Gamma^{\Delta_g}\) is the carré du champ operator (the definition of which is recalled below) associated to \(\Delta_g\), and \(V\) is a \(C^\infty\) function. Most of the notions and results related to this operator \(\Gamma\) can be found in the quite thorough \cite{bgl-book}. 
\begin{defi}[Carré du champ operator]
  Given a differential operator \(L\) on a smooth manifold \(M\), the carré du champ operator \(\Gamma^L\) is a symmetric bilinear map from \(C^\infty(M)\times C^\infty(M)\) onto \(C^\infty(M)\). It is defined by
  \[
  \Gamma^L(a,b)=\frac{1}{2}\pare{L(ab)-aLb-bLa}.
  \]
\end{defi}
An important example to keep in mind is the manifold  \((\R^d,I_d)\), on which the Laplace-Beltrami operator is the usual Laplacian, and its carré du champ operator is, for any smooth functions \(a\) and \(b\),
\[
\Gamma^\Delta(a,b)=\nabla a\cdot\nabla b.
\]
More generally, respectively using the Levici-Vita connection \(\nabla\), and in a local chart, operator \(\Gamma^{\Delta_g}\) is given by
\[
\Gamma^{\Delta_g}(a,b)=g(\nabla a,\nabla b)=g^{ij} \partial_i a \partial_j b.
\]
We iterate this definition to get the second order operator, \(\Gamma_2\).
\begin{defi}[Iterated carré du champ operator]
  Given a differential operator \(L\) on a smooth manifold \(M\), the iterated carré du champ operator \(\Gamma_2^L\) is a symmetric bilinear map from \(C^\infty(M)\times C^\infty(M)\) onto \(C^\infty(M)\) defined by
  \[
  \Gamma_2^L(a,b)=\frac{1}{2}\pare{L(\Gamma^L(a,b))-\Gamma^L(a,Lb)-\Gamma^L(b,La)}.
  \]
\end{defi}
We point out that if \(L=\Delta_g+\Gamma^{\Delta_g}(V,.)\), then \(\Gamma^L=\Gamma^{\Delta_g}\): the drift part of \(L\), namely \(V\), does not appear in this operator. For this reason, and for readability's sake, we shall simply use the notation \(\Gamma\) instead of \(\Gamma^{\Delta_g}\) in what follows, as long as it is not ambiguous. We will also write \(\Gamma(a,a)=\Gamma(a)\) for brievety. We will do the same for \(\Gamma_2\), but one should keep in mind that the drift \emph{does} play a role in \(\Gamma_2\), and \(\Gamma_2^L\not=\Gamma_2^{\Delta_g}\).
The idea behind the carré du champ operator is that it contains all the information about the geometry of the space \((M,g)\), and it thus proves worthwhile to stick to its use in the (sometimes heavy) calculations. Generally, a formula that is valid for the standard Laplacian on \(\R^d\) involving only the \(\Gamma\) operator will remain valid on more general manifolds. For instance, if \(a,b\), and \(c\) are smooth functions defined on \(\R^d\), the Hessian of \(a\) may be expressed in the following fashion
\[
\Hess(a)(\nabla b,\nabla c)=\frac{1}{2}\pare{\nabla b\cdot\nabla(\nabla a\cdot\nabla c)+\nabla c\cdot\nabla(\nabla a\cdot\nabla b)-\nabla a\cdot\nabla(\nabla b\cdot\nabla c)},
\]
hence the following lemma
\begin{lemm}[Hessian of a function]
  Let \(a,b,c\) be smooth functions on  \((M,g)\). Then, the Hessian of \(a\) is given by
  \[
  \nabla^2 a(\nabla b,\nabla c)=\frac{1}{2}\pare{\Gamma(b,\Gamma(a,c))+\Gamma(c,\Gamma(a,b))-\Gamma(a,\Gamma(b,c))},
  \]
  where \(\Gamma\) is the carré du champ operator associated to the Laplace-Beltrami operator \(\Delta_g\).
\end{lemm}
\begin{proof}
  Since the Hessian is symmetric and bilinear, it is sufficient to prove that
  \[
  \nabla^2 a(\nabla b,\nabla b)=\Gamma(b,\Gamma(a,b))-\frac{1}{2}\Gamma(a,\Gamma(b)).
  \]
  By definition,
  \begin{align*}
    \Hess(a)(\nabla b,\nabla b)&= g(\nabla_{\nabla b}\nabla a,\nabla b)\\
    &=\nabla_{\nabla b}g(\nabla a,\nabla b) - g(\nabla a,\nabla_{\nabla b}\nabla b)\\
    &=g(\nabla b,\nabla g(\nabla a,\nabla b))-g(\nabla a, g(\nabla\nabla b,\nabla b))\\
    &=g(\nabla b,\nabla g(\nabla a,\nabla b))-\frac{1}{2}g(\nabla a, \nabla g(\nabla b,\nabla b))
  \end{align*}
  which is the claimed formula.
\end{proof}

In the development of this article, we shall assume so-called curvature-dimension conditions on the different diffusion operators:
\begin{defi}[Curvature-dimension conditions $CD(\rho,n)$]
  A diffusion operator \(L\) is said to satisfy a \(CD(\rho,n)\) condition for \(\rho\in\R\) and \(n \neq 0\), if for every smooth function \(f\),
  \begin{equation}\label{eq:CD(rho,n)}
    \Gamma_2(f)\geq\rho\Gamma(f)+\frac{1}{n}(Lf)^2.
  \end{equation}
\end{defi}
For example, the Bochner-Lichnerowicz formula implies that the Laplace-Beltrami operator $\Delta_g$ satisfies the $CD(\rho,n)$ condition with $n\geq d$  and $\rho\in\R$ whenever the Ricci curvature is uniformly bounded form below by $\rho g$, \cite[Sec.~C.6]{bgl-book}.

Denoting by \(dx\) the Riemannian measure associated to \((M,g)\), we define \(\mu\) to be the reversible measure associated to \(L\), i.e. \(d\mu=Z_Ve^{-V}dx\), where \(Z_V\) a normalizing constant such that \(\mu\) is a probability measure if finite, and \(Z_V=1\) otherwise. The triple $(M,\Gamma,\mu)$ is a Markov triple, as defined in~\cite[Sec.~3.2]{bgl-book}. Then, for any smooth functions \(a,b\) such that the integrals are well defined, we have the following integration by parts formula:
\[
\int_M \Gamma(a,b)d\mu = -\int_M aLbd\mu = -\int_M bLad\mu,
\]
a notable consequence of which is that \(\int_M Lad\mu=0\).
\begin{rmrk}\label{rem-1}
We could have chosen to study a general symmetric Markov semigroup, for instance a  generic Markovian triple \((E,\Gamma,\mu)\) as proposed in~\cite{bgl-book}. However, assuming \(L\) is a symmetric diffusion operator in a smooth Riemannian manifold, which is the case of interest for us, there exists a metric \(\tilde g\) and a function \(V\) such that the operator can be rewritten \(L=\Delta_{\tilde g}+\Gamma^{\Delta_{\tilde g}}(V,.)\), so we are not, in fact, losing any generality.
\end{rmrk}


\section{Weighted inequalities under nonnegative Ricci curvature}\label{sec-3}

For a more pedestrian approach, we first explain the case of the Poincar\'e inequality, only to tackle more general inequalities later on.

\subsection{Weighted Poincaré inequality}\label{sec-3.1}

Let \(\vphi\) be a \(C^2\) positive function on \(M\) such that
\begin{equation}\label{eq:phi_convex}
  \nabla^2\vphi\geq cg
\end{equation}
for some positive constant \(c\). For \(\beta\in\R\) such that \(\vphi^{-\beta}\) is integrable with respect to \(\mu\), let \(\mu_{\vphi,\beta}=Z_{\vphi,\beta}\vphi^{-\beta}\mu\), with the constant \(Z_{\vphi,\beta}\) such that this new measure is a probability measure. The main result is the following:
\begin{thrm}[Weighted Poincaré inequality]\label{th:weighted_poincare}
  Assume that the diffusion operator \(L\) satisfies a \(CD(0,n)\) condition with \(n\geq d\), and fix a real number \(\beta\geq n+1\). Then for all smooth bounded functions $f$,
  \begin{equation}
  \label{eq-18}
    \Var_{\mu_{\vphi,\beta}}(f)=
  \int f^2d\mu_{\vphi,\beta}-\pare{\int fd\mu_{\vphi,\beta}}^{\!\!2}\leq\frac{1}{c(\beta-1)}\int \Gamma(f)\vphi d\mu_{\vphi,\beta}.
  \end{equation}
\end{thrm}
\begin{rmrk}
    As explained in Section~\ref{sec-2}, the main example to keep in mind is the Laplace-Beltrami operator on a Riemannian manifold with a nonnegative Ricci curvature, in which case $\mu_{\vphi,\beta}=Z_{\vphi,\beta}\vphi^{-\beta}dx$. 
    
    This inequality happens to be optimal whenever \((M,g)=(\R^d,\Id)\), and \(\vphi(x)=1+\abs{x}^2\), where the optimal constant is reached for projectors $x\mapsto x_i$, $1\leq i\leq d$. This optimal case has been proved in~\cite{blanchet2007} and also in~\cite{nguyen-14,abj} with  different methods. 
\end{rmrk}
\begin{proof}
Fix \(\beta\in\R\backslash\{2\}\), and define
  \[
  \overline L\coloneqq\vphi L-(\beta-1)\Gamma(\vphi,.).
  \]
 Note that the operator \(\overline L\) is, in fact, of the form \(\overline L=\Delta_{\overline g}+\Gamma^{\Delta_{\overline g}}(\overline V,.)\), where \(\overline g=\vphi^{-1} g\) and \(\overline V=V+(d/2-\beta)\log\vphi\), so the operator we are considering here is obtained from the first one through a conformal transformation. In what follows, everything written with an overline relates to objects associated to the operator \(\overline L\) or the manifold \((M,\overline g)\). For instance, the carré du champ operator is given by \(\overline\Gamma=\vphi\Gamma\), its reversible measure is \(\bar\mu=\mu_{\vphi,\beta}\), and 
  \begin{multline}\label{eq:gamma2}
    \overline\Gamma_2(f)=\Gamma_2^{\ol L}(f)= \vphi^2\Gamma_2(f)+(\beta-1)\varphi\nabla^2\vphi(\nabla f,\nabla f)+\frac{\Gamma(f)}{2}\pare{\vphi L\vphi-(\beta-1)\Gamma(\varphi)}\\+\varphi\Gamma(\varphi,\Gamma(f))-\varphi Lf\Gamma(f,\varphi),
  \end{multline}
  a proof of which can be stringed together with information from \cite[Sec.~6.9.2]{bgl-book}, for instance. 
 
   Now, fix \(f\), a smooth and \emph{bounded} function on \(M\), and consider the Markov semigroup \((f_t)_{t\geq0}\), solution of the initial-value system
  \begin{equation}\label{eq:evolution}
    \begin{cases}
      \partial_t f_t= \overline L f_t &\text{on } (0,+\infty)\times M, \\
      f_0=f &\text{on } M.
    \end{cases}
  \end{equation}
  Consider the variance of \(f_t\) along the flow: 
  \[
  \Lambda(t)\coloneqq \Var_{\bar\mu}(f_t)=\int f_t^2d\bar\mu-\pare{\int f_td\bar\mu}^{\!2}=\int f_t^2d\bar\mu-\pare{\int fd\bar\mu}^{\!2},
  \]
  because \(\ol L\) is mass-preserving. Then, we use the following estimate:
  \begin{lemm}
  \label{lem-1}
    If \(\beta>1\) and \(\beta\neq2\), then for all \(t\geq0\), 
    \begin{equation}\label{eq:entropy1}
      \Lambda''(t)\geq-2c(\beta-1)\Lambda'(t)+\frac{4}{\beta-2}\int\vphi^2\sbra{(\beta-1)\Gamma_2(f_t)-(Lf_t)^2)}d\bar\mu.
    \end{equation}
    Furthermore, there is equality in \eqref{eq:entropy1} for all \(t\geq0\) whenever the Hessian of \(\vphi\) is a constant, i.e. when inequality \eqref{eq:phi_convex} is an equality.  
  \end{lemm}

  Assume that \(L\) satisfies the \(CD(0,n)\) condition for some \(n>0\), and that \(\beta-1\geq n\), and \(\beta>2\). Then, we deduce from equation \eqref{eq:entropy1} that
  \[
  \Lambda''(t)\geq-2c(\beta-1)\Lambda'(t)
  \]
  which we can integrate once betwen $0$ and $t$ to find that
  \[
  -\Lambda'(t)\leq -\Lambda'(0)e^{-2c(\beta-1)t}
  \]
  and then once again, between \(t=0\) and \(t=+\infty\),
  \[
  \Lambda(0)-\lim_{t\to+\infty}\Lambda(t)\leq \frac{-1}{2c(\beta-1)} \Lambda'(0).
  \]
The Markov semigroup studied is ergodic, in other words, it ensures the convergence of \(f_t\) towards its mean in $L^2(\bar\mu)$, 
so that \(\lim_{t\to+\infty}\Lambda(t)=0\), and the previous inequality is simply the stated result, for  smooth and bounded functions, and for $\beta>2$. The general case is established by approximation. The particular case $\beta=2$ (also implying $n=d=1$) is proved by letting $\beta\rightarrow2$ directly in inequality~\eqref{eq-18}.
\end{proof}

\begin{proof}[Proof of Lemma~\ref{lem-1}]
  By definition of \(\bar\mu\), we may integrate by parts the derivative of \(\Lambda\) to find
 $$
  \Lambda'(t)= 2\int f_t\ol Lf_td\bar\mu=-2\int\overline\Gamma(f_t)d\bar\mu,
 $$
  and, differentiating once again, 
$$
    \Lambda''(t)=-4\int\ol\Gamma(f_t,\ol Lf_t)d\bar\mu=-4\int\ol\Gamma(f_t,\ol Lf_t)d\bar\mu+2\int \ol L(\ol\Gamma(f_t))=4\int\overline\Gamma_2(f_t)d\bar\mu. 
 $$
  We may now use formula \eqref{eq:gamma2} to find that
  \begin{multline}\label{eq:lambda''}
    \Lambda''(t)= 4\int(\beta-1)\varphi\nabla^2\vphi(\nabla f_t,\nabla f_t)d\bar\mu\\+4\int \sbra{\vphi^2\Gamma_2(f_t)+\frac{\Gamma(f_t)}{2}\pare{\vphi L\vphi-(\beta-1)\Gamma(\varphi)}+\varphi\Gamma(\varphi,\Gamma(f_t))-\varphi Lf_t\Gamma(f_t,\varphi)}d\bar\mu.
  \end{multline}
  First, the convexity assumption \eqref{eq:phi_convex} on \(\vphi\)  yields
  \[
  4\int(\beta-1)\varphi\nabla^2\vphi(\nabla f_t,\nabla f_t)\geq 4c(\beta-1)\int\overline\Gamma(f_t)d\bar\mu=-2c(\beta-1)\Lambda'(t).
  \]
  Now, in the second integral of equation \eqref{eq:lambda''}, we may not directly use the integration by parts formula, because \(\bar\mu\) is the invariant measure for \(\ol\Gamma\), and not for \(\Gamma\). We must thus rewrite it in terms of \(\mu\). First,
\begin{multline*}
    \int\frac{\Gamma(f_t)}{2}\vphi L\vphi d\bar\mu=Z_{\vphi,\beta}\int\frac{\Gamma(f_t)}{2}\vphi^{1-\beta}L\vphi d\mu =-Z_{\vphi,\beta}\int\Gamma\pare{\vphi,\vphi^{1-\beta}\frac{\Gamma(f_t)}{2}}d\mu \\
    =-\int\frac{\vphi}{2}\Gamma(\vphi,\Gamma(f_t))d\bar\mu-\frac{1-\beta}{2}\int\Gamma(f_t)\Gamma(\vphi)d\bar\mu.
\end{multline*} 
  Then,
\begin{multline*}
    \int\varphi\Gamma(\varphi,\Gamma(f_t))d\bar\mu= Z_{\vphi,\beta}\int\vphi^{1-\beta}\Gamma(\vphi,\Gamma(f_t))d\mu=\frac{Z_{\vphi,\beta}}{2-\beta}\int\Gamma(\vphi^{2-\beta},\Gamma(f_t))d\mu\\
    =\frac{1}{\beta-2}\int\vphi^2L\Gamma(f_t)d\bar\mu,
\end{multline*} 
 and likewise,
\begin{multline*}
    -\int\varphi Lf_t\Gamma(f_t,\varphi)d\bar\mu=\frac{Z_{\vphi,\beta}}{\beta-2}\int Lf_t\Gamma(f_t,\vphi^{2-\beta})d\mu=\\
    \frac{Z_{\vphi,\beta}}{\beta-2}\int \sbra{\Gamma(f_t,\vphi^{2-\beta}Lf_t)- \vphi^{2-\beta}\Gamma(f_t,Lf_t)}d\mu
    =\frac{-1}{\beta-2}\int\vphi^2\sbra{(Lf_t)^2+\Gamma(f_t,Lf_t)}d\bar\mu.
\end{multline*} 
  We conclude putting these three identities together.
\end{proof}


\subsection{\texorpdfstring{\boldmath$\Phi$}{Phi}-entropy and weighted Beckner inequalities}

Instead of the variance, we may consider a generic \(\Phi\)-entropy along the flow. Choose a strictly convex real function \(\Phi\in C^4(I)\), (where $I\subset\R_+^*$ is an open interval) and consider the $\Phi$-entropy of any function $f:M\mapsto I$ such that the integrals below are well defined
\[
\Ent^\Phi_{\mu_{\vphi,\beta}}(f)\coloneqq \int\Phi(f)d\mu_{\vphi,\beta} -\Phi\pare{\int fd\mu_{\vphi,\beta}},
\]
so that \(\Var_{\mu_{\vphi,\beta}}=\Ent^{x\mapsto x^2}_{\mu_{\vphi,\beta}}\). Generalizations of inequalities like \eqref{eq-18} to $\Phi$-entropies have already been studied under $CD(\rho,\infty)$ condition, $\rho>0$ in~\cite{chafai,bolley-gentil}. We find a generalized version of Theorem \ref{th:weighted_poincare}:
\begin{thrm}[\(\Phi\)-entropy inequalities]\label{th:phi_entropy}
  Assume that the diffusion operator \(L\) satisfies a \(CD(0,n)\) condition with \(n\geq0\), and fix a real number \(\beta\geq n+1\).
Let \(\Phi:I\mapsto \R\) be a strictly convex function such that
  \begin{equation}\label{eq:Phi_cond}
    \Phi^{(4)}\Phi''\geq C_{n,\beta}(\Phi^{(3)})^2,
  \end{equation}
  with
  \[
  C_{n,\beta}=\frac{8(\beta-1-n)(2\beta-1)+9n}{8(\beta-1-n)(\beta-1)}.
  \]
  Then, for all smooth bounded functions \(f:M\mapsto  I\), there holds
  \begin{equation}
  \label{eq-17}
  \Ent^\Phi_{\bar\mu}(f)\leq\frac{1}{2c(\beta-1)}\int \Phi''(f)\Gamma(f) \vphi d\bar\mu.
  \end{equation}
\end{thrm}

\begin{proof} Let us first assume that $\beta>2$, the case $\beta=2$ can be proved by passing to the limit in~\eqref{eq-17}. Assume also that $f$ is a smooth and bounded function, the general case can be proved by approximations.
  Consider, just like before, the function \(f_t\), solution of the initial-value system  \eqref{eq:evolution} starting from $f$, and define, for \(t\geq0\),
  \[
  \Lambda(t)=\Ent^\Phi_{\bar\mu}(f_t)=\int\Phi(f_t)d\bar\mu-\Phi\pare{\int f_td\bar\mu}=\int\Phi(f_t)d\bar\mu-\Phi\pare{\int fd\bar\mu},
  \]
  where, again, \(\bar\mu=\mu_{\vphi,\beta}\). Differentiating the entropy yields
  \[
  \Lambda'(t)=-\int\frac{\ol\Gamma(\Phi'(f_t))}{\Phi''(f_t)}d\bar\mu,
  \]
and, from~\cite[Lem.~4]{bolley-gentil},
  \[
  \Lambda''(t)=\int\sbra{2\frac{\ol\Gamma_2( \Phi'(f_t))}{\Phi''(f_t)}+\pare{\frac{-1}{\Phi''}}''(f_t)\pare{\frac{\ol\Gamma(\Phi'(f_t))}{\Phi''(f_t)}}^2}d\bar\mu.
  \]
  We follow exactly the same steps as in Section~\ref{sec-3.1}. For brievety, we write \(h\coloneqq\Phi'(f_t)\). First, expanding \(\ol\Gamma_2\) in terms of \(\Gamma_2\), \(\Gamma\) and \(L\), and also using the convexity hypothesis on \(\vphi\) \eqref{eq:phi_convex}, we find, 
  \begin{multline*}
    \Lambda''(t)\geq-2c(\beta-1)\Lambda'(t)+\int\left[\frac{2}{\Phi''(f_t)}\left(\vphantom{\pare{\frac{\Gamma(h)}{\Phi''(f)}}^2} \vphi^2\Gamma_2(h)+\frac{\Gamma(h)}{2}(\vphi L\vphi-(\beta-1)\Gamma(\vphi))+\vphi\Gamma(\vphi,\Gamma(h))\right.\right.\\
      \left.\left.\vphantom{\frac{\Gamma(h)}{2}}-\vphi Lh\Gamma(h,\vphi)\right)  +\vphi^2\pare{\frac{-1}{\Phi''}}''(f_t)\pare{\frac{\Gamma(h)}{\Phi''(f_t)}}^2\right]d\bar\mu.
  \end{multline*}
  We may now give these terms the same treatment as in the previous section: the goal is to remove all derivatives on \(\vphi\). The calculations are not made explicit here; they involve the exact same ingredients we used before, only with more terms appearing. We finally find, using only integration by parts, that 
  \[
  \Lambda''(t)\geq-2c(\beta-1)\Lambda'(t)+\frac{1}{\beta-2}\int\frac{\vphi^2}{\Phi''(f_t)}\pare{a_0\Gamma_2(h)+a_0'(Lh)^2+a_2\Gamma(h)Lh+a_3\Gamma(h)^2}d\bar\mu,
  \]
  where
  \[
  \left\{\begin{aligned}
  a_0&=2(\beta-1),\\
  a_0'&=-2,\\
  a_2&=3\frac{\Phi'''(f_t)}{(\Phi''(f_t))^2},\\
  a_3&=(\beta-1)\frac{\Phi^{(4)}(f_t)}{(\Phi''(f_t))^3}+(1-2\beta)\frac{(\Phi'''(f_t))^2}{(\Phi''(f_t))^4}.
  \end{aligned}\right.
  \]
  Just like for the Poincaré inequality proof, this inequality becomes an equality when \(\Hess(\vphi)=cg\). Invoke the \(CD(0,n)\) condition to find that 
  \begin{equation}\label{eq:pol_ineq}
    \Lambda''(t)\geq-2c(\beta-1)\Lambda'(t)+\frac{1}{\beta-2}\int\frac{\vphi^2}{\Phi''(f)}\pare{a_1(Lh)^2+a_2\Gamma(h)Lh+a_3\Gamma(h)^2}d\bar\mu,
  \end{equation}
  where \(a_1=a_0/n+a_0'=\frac{2}{n}(\beta-1-n)\). In the same way as before, we want this integrated quantity to be nonnegative. Since \(\Phi\) is strictly convex, it is sufficient to require the polynomial function \(X\mapsto a_1X^2+a_2X+a_3\) to be nonnegative, which itself is equivalent to
  \[
  a_1\geq 0 \quad\text{and}\quad a_2^2-4a_1a_3\leq0.
  \]
  Straightforward computation yields that this, in turn, is equivalent, whenever \(n\geq0\), to
  \[
  \beta\geq n+1\quad\text{and}\quad \Phi^{(4)}\Phi''\geq C_{\beta,n}(\Phi''')^2,
  \]
  where
  \[
  C_{\beta,n}=\frac{8(\beta-1-n)(2\beta-1)+9n}{8(\beta-1-n)(\beta-1)}.
  \]
  If this condition is satisfied, the integrand is pointwise nonnegative, thus
  \[
  \Lambda''(t)\geq-2c(\beta-1)\Lambda'(t),
  \]
  and integrating this twice yields the theorem.
\end{proof}

We can explicit this theorem in the particular case where \(\Phi(X)=X^p\), \(p> 1\). Condition \eqref{eq:Phi_cond} is then equivalent to
\[
p\in[p^*,2],
\]
where
\begin{equation}\label{eq:pstar}
  p^*=1+\frac{8(\beta-1-n)+9n}{8\beta(\beta-1-n)+9n}\in(1,2].
\end{equation}
Thus, if this condition is satisfied,
\[
\int f^p d\bar\mu-\pare{\int fd\bar\mu}^p\leq\frac{p(p-1)}{2c(\beta-1)}\int f^{p-2}\vphi\Gamma(f) d\bar\mu.
\]
Rewriting this inequality with \(\tilde f= f^{2/p}\) in place
of \(f\), we find
\begin{coro}[Weighted Beckner inequalities]\label{th:weighted_beckner}
Under the same assumptions as Theorem~\ref{th:phi_entropy}, for all \(p\in\sbra{p^*,2}\), and all smooth bounded functions $f$, we have
  \begin{equation}\label{eq:weighted_beckner}
    \frac{p}{p-1}\sbra{\int f^2d{\mu_{\vphi,\beta}}-\pare{\int f^{2/p}d{\mu_{\vphi,\beta}}}^p } \leq \frac{2}{c(\beta-1)}\int\vphi\Gamma (f)d{\mu_{\vphi,\beta}},
  \end{equation}
  where \(p^*\) is given by \eqref{eq:pstar}.
\end{coro}
Corollary \ref{th:weighted_beckner} is optimal (and thus so is theorem \ref{th:phi_entropy}) in the sense that there exists no constant  \(0<C<(c(\beta-1))^{-1}\) such that \(\ol L\) satisfies a Beckner inequality \(B_p(C)\) (see Definition \ref{bpc} below), because \(B_p(C)\) (for any \(p\in(1,2)\)) implies the Poincaré inequality with constant \(C\). Indeed, testing inequality \eqref{eq:weighted_beckner} with the function  \(1+\eps f \), $f$ bounded, we find that
\[
\frac{p}{p-1}\sbra{\int (1+\eps f)^2d{\bar\mu}-\pare{\int (1+\eps f)^{2/p}d{\bar\mu}}^p } \leq \frac{2\eps^2}{c(\beta-1)}\int\vphi\Gamma (f)d{\bar\mu},
\]
which, after being expanded when \(\eps\) is small, turns into
\[
\int f^2d\bar\mu - \pare{\int fd\bar\mu}^2+o(1)\leq\frac{1}{c(\beta-1)}\int \vphi\Gamma(f)d\bar\mu,
\]
which is exactly the optimal weighted Poincaré inequality \ref{th:weighted_poincare}.

\begin{rmrk}\label{rk:refined_ineq}
  In the case of interest, when \(\Phi(X)=X^p\), we may explicit inequality \eqref{eq:pol_ineq}
  \[
  \Lambda''(t)\geq-2c(\beta-1)\Lambda'(t)+\frac{1}{\beta-2}\int\vphi^2h^{\frac{2-p}{p-1}}\pare{a'_1(Lh)^2+a'_2\frac{\Gamma(h)}{h}Lh+a'_3\frac{\Gamma(h)^2}{h^2}}d\bar\mu,
  \]
  where the \(a_i'\), \(i\in\{1,2,3\}\) are real constants.
  The inequality can be then improved using the fact that
  \[
  a'_1X^2+a'_2X+a'_3 = a'_1\pare{X+\frac{a'_2}{2a'_1}}^2-\frac{(a'_2)^2-4a'_1a'_3}{4a'_1},
  \]
  so that, whenever \(\delta\coloneqq \frac{-1}{4a'_1}((a'_2)^2-4'a_1a'_3) \geq 0\), we find that
  \begin{align*}
    \Lambda''(t)&\geq-2c(\beta-1)\Lambda'(t)+\frac{\delta}{\beta-2}\int\vphi^2h^{\frac{2-p}{p-1}}\frac{\Gamma(h)^2}{h^2}d\bar\mu\\
    &\geq -2c(\beta-1)\Lambda'(t)+\frac{\delta}{\beta-2}\frac{\pare{\int\vphi h^{\frac{2-p}{p-1}}\Gamma(h)d\bar\mu}^2}{\int h^{\frac{p}{p-1}}d\bar\mu}\\
    &= -2c(\beta-1)\Lambda'(t)+ C_p\frac{\Lambda'(t)^2}{\Lambda(t)},
  \end{align*}
  by Jensen's inequality, where, for reference,
  \[
  C_p=\frac{\delta}{\beta-2}(p-1)^2p^{\frac{2-p}{p-1}}\geq 0.
  \]
  This leads, when integrated, to a refined version of Beckner's inequality, which we will come back to in section \ref{se:CD(rho,n)}, and specifically, corollary \ref{th:improved_beckner}.
\end{rmrk}
\begin{rmrk}
As explained in Section~\ref{sec-1}, we extend the result of Nguyen~\cite{18nguyen} in two aspects. First, Beckner 
inequalities~\eqref{eq:weighted_beckner} are proved in the more general 
context of Riemannian manifolds satisfying a $CD(0,n)$ condition, and secondly, the range of parameter $p\in[p^*,2]$ given by~\eqref{eq:pstar} contains strictly the one proposed in~\cite{18nguyen}.

On the other hand, in \cite{bgs}, corollary \ref{th:weighted_beckner} is proved in the special case of \(M=\R^d\) and \(\vphi(x)=1+\abs{x}^2\). Interestingly, the range found in their article is greater than what we can manage here. Indeed, it is valid for all \(p\in\sbra{p^*_{BGS},2}\), where
  \[
  1<p^*_{BGS}=1+\frac{1}{\beta-d} < p^*.
  \]
  It might be worth it to note that in our case, we used the fact that the second degree polynomial appearing in the integral is greater than \(0\), when this is in fact quite a gross lower bound. Indeed, the argument is that
  \[
  \int\vphi^2h^{\frac{2-p}{p-1}}\pare{\frac{hLh}{\Gamma (h)}+\frac{a'_2}{2a'_1}}^{\mkern-5mu 2}\frac{\Gamma(h)^2}{h^2}d\bar\mu\geq 0,
  \]
  but the squared term can probably be controlled in a way that leads to a wider range of \(p\), since
  \[
  \int \sbra{hLh+\frac{a'_2}{2a'_1}\Gamma(h)}d\mu=\pare{\frac{a'_2}{2a'_1}-1}\int\Gamma(h)d\mu<0
  \]
  whenever \(h\not\equiv0\).
\end{rmrk}


\section{Spaces with positive curvature and real dimension}\label{se:CD(rho,n)}

The family of inequalities considered in this article, especially for sections~\ref{se:CD(rho,n)} and~\ref{sec-5}, is the following interpolation between the Poincaré inequality and the logarithmic Sobolev inequality
\begin{defi}[Beckner inequalities]
\label{def-1}
  The Markov triple \((M,\Gamma,\mu)\) is said to satisfy a Beckner inequality \(B_p(C)\) with parameter \(p\in\left(1,2\right]\) and constant \(C>0\) if, for all nonnegative smooth bounded functions \(f\),
    \begin{equation}\label{eq:beckner}
      \frac{p}{p-1}\pare{\int f^2d\mu-\pare{\int f^{2/p}d\mu}^{\!\!p} } \leq 2C\int\Gamma (f)d\mu.
    \end{equation}
\end{defi}
The constants in front of the integrals are chosen so that for \(p=2\), this is exactly the Poincaré inequality with constant \(C\), and the limiting case \(p\to 1\) corresponds to the logarithmic Sobolev inequality, again with constant \(C\). Indeed,
\[
\lim_{p\to1}\frac{p}{p-1}\pare{\int f^2d\mu-\pare{\int f^{2/p}d\mu}^{\!\!p} }=\int f^2\log\pare{\frac{f^2}{\int f^2d\mu}}d\mu.
\]

\begin{rmrk}
    The weighted Beckner inequalities proved in the previous section can be seen as classical Beckner inequalities. Indeed, theorem \ref{th:weighted_beckner} states that \(B_p((c(\beta-1))^{-1})\) is valid for the triple \((M,\vphi\Gamma,\mu_{\vphi,\beta})\).
\end{rmrk}

In this section, we consider a diffusion operator \(L\) defined on \((M,g)\), juste like in Section~\ref{sec-3}, with $\mu$ being its reversible measure. We assume that $L$ satisfies a curvature-dimension condition $CD(\rho,n)$,  with $\rho>0$ and $n\in\R$.
That means, specifically but not exclusively, that we will consider negative \(n\). Such spaces, sometimes referred to as having negative \emph{effective dimension}, have been studied in the past, with articles dating as far back as 2003, up to our knowledge \cite{scheffer,milman,milman18,bgs}. Furthermore, one can easily construct examples of such operators with~\cite[Cor. 4.13]{ohta}.
Writing down the \(CD(\rho,n)\) inequality in good local coordinates like in \cite{bgs}, one can check that a necessary criterion for the inequality to be true is that \(d/n\leq 1\), which means that we are actually restricted to \(n\in\R\backslash\sbra{0,d}\). The main difference, compared to the previous section, is that the curvature is \emph{positive}, which is a stronger assumption, but \(n\) can be negative, which is a weaker assumption. 

Let us first recall the well known result for the Poincar\'e inequality, the proof of which is given in~\cite[Thm.~4.8.4]{bgl-book}.
\begin{thrm}[Poincar\'e inequality under $CD(\rho,n)$]\label{th:poincare}
Assume that the diffusion operator \(L\) satisfies a \(CD(\rho,n)\) condition, with \(\rho>0\) and \(n\in\R\backslash[0,d)\). Then the following Poincar\'e inequality holds, 
\begin{equation}
\label{eq-poin2}
\Var_{\mu}(f)\leq\frac{n-1}{\rho n}\int \Gamma(f)d\mu.
\end{equation}
\end{thrm}
We extend this result to Beckner inequalities in the following theorem:
\begin{thrm}[Beckner inequalities under $CD(\rho,n)$]\label{th:beckner}
  Assume that the diffusion operator \(L\) satisfies a \(CD(\rho,n)\) condition, with \(\rho>0\) and \(n\in\R\backslash[-2,d)\). Then, \(B_p\pare{\frac{n-1}{\rho n}}\) is satisfied
    \begin{itemize}
    \item for all \(p\in \left(1,2\right]\) if \(n\geq d\),
    \item for all \(p\in \left[p^*,2\right]\) if \(n< -2\), where
      \[
      p^*=1+\frac{1-4n}{2n^2+1}.
      \]
    \end{itemize}
\end{thrm}
Interestingly, we find nothing  for \(n\in\left[-2,0\right)\), which corresponds to the weakest \(CD(\rho,n)\) conditions possible, even though the Poincaré inequality remains valid in that range. It seems like there is not enough structure in that case. 

\begin{proof} 
  In the same fashion as the previous section, fix \(f\), a bounded smooth nonnegative function on \(M\), and consider the function \(f_t\), solution of the initial-value system
  \begin{equation}\label{eq:evolution2}
    \begin{cases}
      \partial_t f_t= L f_t &\text{on } (0,+\infty)\times M, \\
      f_0=f &\text{on } M,
    \end{cases}
  \end{equation}
  and consider its \(\Phi\)-entropy along the flow, where \(\Phi\) is assumed to be strictly convex.
  \[
  \Lambda(t)=\Ent^\Phi_{\mu}(f_t)=\int\Phi(f_t)d\mu-\Phi\pare{\int f_td\mu},
  \]
  Invoking  once again,
  \[
  \Lambda'(t)=-\int\frac{\Gamma(\Phi'(f_t))}{\Phi''(f_t)}d\mu,\quad\Lambda''(t)=\int\sbra{2\frac{\Gamma_2( \Phi'(f_t))}{\Phi''(f_t)}+\pare{\frac{-1}{\Phi''}}''\!\!\!(f_t)\pare{\frac{\Gamma(\Phi'(f_t))}{\Phi''(f_t)}}^{\!\!2}}d\mu.
  \]
  Classically, one can assume \(-1/\Phi''\) to be convex, whence the \(CD(\rho,n)\) condition yields
  \[
  \Lambda''(t)\geq\int\frac{2}{\Phi''(f_t)}\sbra{\rho\Gamma(\Phi'(f_t)) + \frac{1}{n}(L\Phi'(f_t))^2}d\mu,
  \]
  and, when \(n>0\), this leads to
  \[
  \Lambda''(t)\geq -2\rho\Lambda'(t),
  \]
  which, ultimately, proves Poincaré-type inequalities, using arguments like the ones in section~\ref{sec-3}. The problem is that the constant appearing in the inequality is, in fact, not optimal, and furthermore, the argument fails whenever \(n<0\). For the sake of simplicity, we will now assume that \(\Phi(X)=X^p\), with \(p\in(1,2]\), but the argument could very well be generalized to more general \(\Phi\). The derivatives of the entropy then become
  \[
  \Lambda'(t)=-\frac{p}{p-1}\int f_t^{2-p}\Gamma(f_t^{p-1}) d\mu,\quad\Lambda''(t)=\frac{p}{p-1}\int\sbra{2f_t^{2-p}\Gamma_2( f_t^{p-1})+\frac{2-p}{p-1}f_t^{4-3p}\Gamma(f_t^{p-1})^2}d\mu.
  \]
  Rewriting these quantities with respect to \(q\coloneqq\frac{2-p}{p-1}\in\left[0,+\infty\right)\), and \(h=f_t^{p-1}\), 
  \[
  \frac{p-1}{p}\Lambda'(t)=-\int h^q\Gamma(h) d\mu,\quad\frac{p-1}{p}\Lambda''(t)=\int\sbra{2h^q\Gamma_2(h)+qh^{q-2}\Gamma(h)^2}d\mu.
  \]
  To get the most information out of this, we shall apply the \(CD(\rho,n)\) condition not to the function \(h\), but to \(\eta(h)\), where \(\eta\) is some function to be chosen later. Expanding every term in the \(CD(\rho,n)\) inequality~\eqref{eq:CD(rho,n)}
  \[
  \Gamma_2(\eta(h))\geq\rho\Gamma(\eta(h))+\frac{1}{n}(L\eta(h))^2 
  \]
  yields
  \[
  \eta'^2(h)\Gamma_2(h)+\eta'(h)\eta''(h)\Gamma(h,\Gamma(h))+(\eta''(h)\Gamma(h))^2\geq\rho\eta'^2(h)\Gamma(h)+\frac{1}{n}(\eta'(h)Lh+\eta''(h)\Gamma(h))^2.
  \]
  This inequality is, in particular, true for the power function \(\eta(x)=x^{\theta+1}\), with \(\theta\in\R\), so that
  \[
  \Gamma_2(h)\geq \underbrace{\rho\Gamma(h)}_{i}
  +2\theta\underbrace{\frac{1}{n}\frac{\Gamma(h)Lh}{h}}_{ii}
  -\theta\underbrace{\frac{\Gamma(h,\Gamma(h))}{h}}_{iii}
  +\theta^2\underbrace{\pare{\frac{1-n}{n}}\frac{\Gamma(h)^2}{h^2}}_{iv}
  +\underbrace{\frac{1}{n}(Lh)^2}_{v}.
  \]
  Multiplying this inequality by \(h^q\) and then integrating it, we are left with five terms to consider. The first term corresponds to the first derivative of the entropy, \(\Lambda'(t)\), so we may leave it as it is. Terms \(ii\) and \(iii\) are trickier, because their sign is not known, so we must take care of them. Term \(iv\) is always negative since \(n\notin\sbra{0,1}\), but it is compensated by another (positive) term of the same nature appearing naturally in \(\Lambda''(t)\). The sign of the last term depends on the sign of \(n\), so we must take care of it as well, at least for negative \(n\).

  \emph{Term} \(ii\): an integration by parts yields
  \[
  \int h^{q-1}\Gamma(h)Lhd\mu=-\int \sbra{h^{q-1}\Gamma(h,\Gamma(h))+(q-1)h^{q-2}\Gamma(h)^2}d\mu.
  \]

  \emph{Term} \(iii\): we do not do anything for now with this term, and will adjust \(\theta\) later so that it disappears.

  \emph{Term} \(v\): we use the fact that, by definition,
  \[
  \int\Gamma_2(h)d\mu=\int\sbra{\frac{1}{2}L(\Gamma(h))-\Gamma(h,Lh)}d\mu=\int(Lh)^2d\mu
  \]
  to prove that, for any real function \(\eta\),
  \[
  \int\eta'^2(h)(Lh)^2d\mu=\int\sbra{\eta'^2(h)\Gamma_2(h)+3\eta(h)\eta'(h)\Gamma(h,\Gamma(h))+2(\eta''^2(h)+\eta'(h)\eta'''(h))\Gamma(h)^2}d\mu.
  \]
  In particular,
  \[
  \int h^q(Lh)^2d\mu=\int \sbra{h^q\Gamma_2(h)+q(q-1)h^{q-2}\Gamma(h)^2+\frac{3}{2}qh^{q-1}\Gamma(h,\Gamma(h))}d\mu.
  \]
  Finally, we are left with an equality still involving the parameter \(\theta\):
  \[
  \int h^q\Gamma_2(h)d\mu\geq \frac{\rho n}{n-1}\int h^q\Gamma(h)d\mu+ \int \sbra{A h^{q-1} \Gamma(h,\Gamma(h))+Bh^{q-2}\Gamma( h)^2 }d\mu
  \]
  where
  \[
  \left\{
  \begin{aligned}
    A&=\frac{1}{n-1}\left(\frac{3q}{2}-\theta(n+2)\right),\\
    B&=\frac{q(q-1)}{n-1}-\theta^2-2\theta\frac{q-1}{n-1}.
  \end{aligned}\right.
  \]
  Choosing \(\theta\) so that \(A=0\), i.e. \(\theta=\frac{3q}{2(n+2)}\), we find that
  \[
  \int\sbra{ 2f^q\Gamma_2(h)+qh^{q-2}\Gamma(h)^2}d\mu\geq \frac{2\rho n}{n-1}\int h^q\Gamma(h)d\mu+\alpha \int h^{q-2}\Gamma(h)^2d\mu,
  \]
  with 
  \begin{align*}
    \alpha=2B+q&=\frac{q}{2(n+2)^2}(q(4n-1)+2n(n+2))\\
    &=\frac{1}{2(n+2)^2}\pare{\frac{2-p}{p-1}}\pare{\frac{2-p}{p-1}(4n-1)+2n(n+2)}
  \end{align*}
  Remembering that \(q\geq0\), this constant \(\alpha\) turns out to be nonnegative for the following range of parameters
  \[
  \begin{cases}
    n\geq d\\
    q\geq0
  \end{cases}
  \quad\quad\text{or}\quad\quad\quad
  \begin{cases}
    n< -2\\
    q\in\sbra{0,q^*}
  \end{cases}
  \text{with }q^*=\frac{2n(n+2)}{1-4n},
  \]
  or equivalently, in terms of the exponent \(p\),
  \[
  \begin{cases}
    n\geq d\\
    p\in\left(1,2\right]
  \end{cases}
  \quad\quad\text{or}\quad\quad\quad
  \begin{cases}
    n< -2\\
    p\in\sbra{p^*,2}
  \end{cases}
  \text{with }p^*=1+\frac{1-4n}{2n^2+1},
  \]
  Whenever \(\alpha\geq0\), we may, at last, compare \(\Lambda''\) to \(\Lambda'\). Indeed, we find that
  \[
  \Lambda''(t)\geq - \frac{2\rho n}{n-1}\Lambda'(t),
  \]
  which proves the claimed \(B_p\pare{\frac{n-1}{\rho n}}\) inequality when integrated twice.
\end{proof}

To prove this theorem, we used the nonnegativity of a specific term in a differential inequality, but in fact, we can do a little bit better and compare it to the other terms, in order to prove a refined version of the Beckner inequalities we are considering.
\begin{defi}\label{bpc}
  The Markov triple \((M,\Gamma,\mu)\) is said to satisfy a refined Beckner inequality \(B_p^*(C,\theta)\) with parameter \(p\in\left(1,2\right]\) and constants \(C>0\) and \(\theta\geq0\) if, whenever \(\theta\neq1\), 
    \[
    \frac{p}{p-1}\pare{\frac{1}{1-\theta}}\pare{\int f^2d\mu-\pare{\int f^{2/p}d\mu}^{\!\!(1-\theta)p}\pare{\int f^2d\mu}^{\!\!\theta}} \leq 2C\int\Gamma (f)d\mu
    \]
    for all smooth functions \(f\), and, when \(\theta=1\),
    \[
    \frac{p}{p-1}\pare{\int f^2d\mu}\log\pare{\frac{\int f^2d\mu}{\pare{\int f^{2/p}d\mu}^p}}\leq 2C\int\Gamma (f)d\mu.
    \]
\end{defi}
With this definition, \(B_p(C)\) is the same as \(B_p^*(C,0)\). The inequality given for \(\theta=1\) is simply the limit of the other inequality when \(\theta\to 1\).
This indeed corresponds to an improved version of the Beckner inequality, because for all \(x,y>0\) and \(\theta\in\R_+\backslash\{1\}\),
\[
\frac{x-y^{1-\theta}x^\theta}{1-\theta} \geq x-y,
\]
and more generally, \(B_p^*(C,\theta)\) implies \(B_p^*(C,\theta')\) for all \(\theta'\in\sbra{0,\theta}\).
Such improvements have been shown in~\cite{arnolddolbeault05,bolley-gentil} under the $CD(\rho,\infty)$ condition. The limit case, that is, for the usual entropy, is proposed in~\cite[Thm.~6.8.1]{bgl-book}.

\begin{thrm}[Improved Beckner inequalities]\label{th:improved_beckner}
  Under the same assumptions of in Theorem \ref{th:beckner}, the inequality \(B_p^*\pare{\frac{n-1}{\rho n},\theta}\) is satisfied for the same range of parameter \(p\), and for \(\theta\) given by
  \[
  \theta=\frac{1}{2(n+2)^2}\pare{\frac{p}{p-1}}\pare{\frac{2-p}{p-1}}\pare{\frac{2-p}{p-1}(4n-1)+2n(n+2)}.
  \]
\end{thrm}

\begin{proof}
  Coming back to the proof of theorem \ref{th:beckner}, we have that
  \[
  \int\sbra{ 2f^q\Gamma_2(h)+qh^{q-2}\Gamma(h)^2}d\mu\geq \frac{2\rho n}{n-1}\int h^q\Gamma(h)d\mu+\alpha \int h^{q-2}\Gamma(h)^2d\mu,
  \]
  or, in terms of \(\Lambda\),
  \[
  \Lambda''(t) \geq - \frac{2\rho n}{n-1}\Lambda'(t)+\alpha\frac{p}{p-1} \int h^{q-2}\Gamma(h)^2d\mu.
  \]
  Invoking Jensen's inequality, we find that
  \begin{align*}
    \int h^{q-2}\Gamma(h)^2d\mu &\geq \frac{\pare{\int h^q\Gamma(h)d\mu}^2}{\int h^{q+2}d\mu}\\
    &=\pare{\frac{p-1}{p}}^{\!\!2}\frac{\Lambda'(t)^2}{\Lambda(t)},
  \end{align*}
  so that 
  \[
  \Lambda''(t)\geq - \frac{2\rho n}{n-1}\Lambda'(t)+\alpha\frac{p}{p-1}\frac{\Lambda'(t)^2}{\Lambda(t)},
  \]
  just like we found in remark \ref{rk:refined_ineq}. Writing \(\theta=\alpha\frac{p}{p-1}\), we may now integrate this inequality, to find that
  \[
  -\frac{\Lambda'(t)}{\Lambda(t)^\theta} \leq -\frac{\Lambda'(0)}{\Lambda(0)^\theta}\exp\pare{- \frac{2\rho n}{n-1}t},
  \]
  which, integrated once more between \(0\) and \(+\infty\), leads to the claimed inequality.
\end{proof}

\begin{rmrk}

  As it turns out, the operator \(\ol L\) from last section does not verify a good enough \(CD(\rho,n)\) condition. The best constant \(c(\beta-1)\) only arises using an integrated \(CD(\rho,n)\) criterion.
  To see this, we can use the following result from \cite{bakrystflour}: the operator \(L=\Delta_g+\Gamma(V,.)\), defined on \((M,g)\),  satisfies a \(CD(\rho,n)\) condition if, and only if,
  \[
  \frac{n-d}{n}\pare{\Ric_g-\nabla^2V-\rho g}\geq\frac{1}{n}\nabla V\otimes\nabla V,
  \]
  this tensorial reformulation being valid for any \(\rho\in\R\) and \(n\not\in\sbra{0,d}\).
  As stated in the proof to theorem \ref{th:weighted_poincare}, the operator  \(\ol Lf=\vphi L f-(\beta-1)\Gamma(\vphi,f)\) on \((M,g)\) is related to a Laplace-Beltrami operator through the conformal transformation with conformal factor \(\vphi^{-1}\). Thus, writing \(\overline L=\Delta_{\overline g}+\Gamma^{\Delta_{\overline g}}(\overline V,.)\), for some explicit function \(V\), we find a somewhat easier to verify criterion for the \(CD(\rho,n)\) condition in \((M,\ol g)\):
  \begin{equation}\label{eq:tensorial_criterion}
  \Ric_g+(\beta-1)\frac{\nabla^2\vphi}{\vphi}+
  \pare{2-d-\frac{(2\beta-1)^2}{n-d}}\frac{\nabla\vphi\otimes\nabla\vphi}{4\vphi^2}
  +\pare{\frac{\Delta\vphi}{2\vphi}-2\beta\frac{\Gamma(\vphi)}{4\vphi^2}-\frac{\rho}{\vphi}}g\geq 0.
  \end{equation}
  We leave it to the courageous to verify that indeed, even in the case where everything is nice and explicit, for instance for \(\vphi(x)=1+\abs{x}^2\), there exists no couple \((\rho,n)\in\R_+^*\times\pare{\R\backslash{\sbra{0,d}}}\) such that inequality \eqref{eq:tensorial_criterion} is verified and
  \[
  \frac{\rho n}{n-1}=c(\beta-1),
  \]
  making theorem \ref{th:weighted_poincare} truly an \emph{integrated} \(CD(\rho,n)\) criterion result.
\end{rmrk}


\section{Results on the real line}
\label{sec-5}

In dimension \(d=1\), simplifications happen, so that we are able to do calculations directly. All manifolds of dimension \(1\) are conformal to \((\R,1)\), and even though the study of a generic manifold \((\R,g)\) does not reduce exactly to the study of \((\R,1)\), the calculations are similar, and so we will only consider the Lebesgue measure \(\lambda\) for the reference measure.
Following the work in section \ref{sec-3}, assume \(\vphi\) is a \(C^2\), positive, convex function such that
\[
\vphi''\geq c 
\]
for some constant \(c>0\). Fix \(\beta\in\R\) such that \(\vphi^{-\beta}\) is in \(L^1(\R,\lambda)\). Then, we find the following result version of the Poincaré inequality, for the probability measure \(\mu_{\vphi,\beta}=Z_{\vphi,\beta}\vphi^{-\beta}\lambda\)
\begin{thrm}\label{th:weighted_poincare_1d}
  Fix a real number \(\beta>1\). Then for all smooth bounded functions \(f\), 
  \begin{equation}\label{eq:poincare_1d}
    \int f^2d\mu_{\vphi,\beta}-\pare{\int fd\mu_{\vphi,\beta}}^2\leq\frac{1}{c(\beta-1)}\int (f')^2\vphi d\mu_{\vphi,\beta}.
  \end{equation}
\end{thrm}
For \(\beta>2\), this is in fact theorem \ref{th:weighted_poincare}, so this theorem is an extension to smaller exponants \(\beta\).

\begin{proof}
  The idea is to apply Theorem \ref{th:beckner} to the special case of the operator \(\ol L\) defined by
  \[
  \ol Lf = \vphi f''-(\beta-1)\vphi' f',
  \]
  after proving it satisfies a \(CD(\rho,n)\) condition with negative dimension. Since we are working in dimension \(1\), expliciting the \(CD(\rho,n)\) is not too hard, and that is just what we do. Indeed, straightforward computations yield
  \begin{align*}
    \ol\Gamma_2(f)&=\frac{1}{2}\ol L(\ol \Gamma (f))-\ol\Gamma(f,\ol Lf)\\
    &=\frac{1}{2}\pare{(2\beta-1)\vphi\vphi''+(1-\beta)(\vphi')^2}(f')^2 + \vphi\vphi' f'f'' +\vphi^2 (f'')^2,
  \end{align*}
  so that, given \(\rho\geq 0\) and \(n\in\R\backslash\sbra{0,1}\),
  \[
  \ol\Gamma_2(f)\geq \rho\ol\Gamma(f)+\frac{1}{n}(\ol Lf)^2
  \]
  for all smooth functions \(f\) if, and only if,
  \[
  A(f')^2+Bf'f''+C(f'')^2 \geq 0
  \]
  for all \(f\), where
  \[
  \left\{
  \begin{aligned}
    A&=\frac{1}{2}\pare{(2\beta-1)\vphi\vphi''+(1-\beta)(\vphi')^2} -\rho\vphi-\frac{1}{n}(1-\beta)^2(\vphi')^2,\\
    B&= \frac{1}{n}(n+2\beta-2)\vphi\vphi',\\
    C&= \pare{1-\frac{1}{n}}\vphi^2.
  \end{aligned}
  \right.
  \]
  since \(C\geq 0\), this is in turn equivalent to \(B^2-4AC\leq 0\), which, after simplifications, boils down to the condition
  \begin{equation}\label{eq:CD_condition_1d}
    -\pare{\beta-\frac{1}{2}}\pare{\frac{1}{n-1}\pare{\beta-\frac{1}{2}}+\frac{1}{2}}\frac{(\vphi')^2}{\vphi}+\pare{\beta-\frac{1}{2}}\vphi''-\rho \geq 0.
  \end{equation}
  For \(n=2(1-\beta)\), the first term in the above inequality disappears and the condition becomes
  \[
  \pare{\beta-\frac{1}{2}}\vphi''-\rho \geq 0,
  \]
  which is clearly true for \(\rho=c\pare{\beta-\frac{1}{2}}\), which proves that \(\ol L\) satisfies the \(CD(c\pare{\beta-\frac{1}{2}},2(1-\beta))\) condition. We may now apply the theorem \ref{th:poincare} to conclude that the Poincaré inequality is valid with constant
  \[
  \frac{n-1}{\rho n}=\frac{1-2\beta}{c(1-\beta)(2\beta-1)}=\frac{1}{c(\beta-1)},
  \]
  which is the claimed result.
\end{proof}
This proof proves more than just the Poincaré inequality, since theorem \ref{th:beckner} provides a range of exponants for which the Beckner inequality holds, but this is only true when \(n<-2\), which corresponds to \(\beta>2\). As far as Beckner inequalities go, this result is exactly the same as the one in section \ref{sec-3}, and as such only constitutes an example of application of the results in section \ref{se:CD(rho,n)}. The interest, however, lies in the fact that we extend the range of \(\beta\) for which the Poincaré inequality is valid.

As it turns out, this inequality is optimal for \(\beta \geq 3/2\), but it is not optimal anymore for \(\beta\in\left[1,3/2\right)\), as proved in \cite{joulin} for the function \(\vphi(x)=1+x^2\). In fact, they find that inequality \eqref{eq:poincare_1d} is valid for \(\beta\in\left(1/2,3/2\right]\), and the optimal constant in that range changes from \((2(\beta-1))^{-1}\) to \((\beta-1/2)^{-2}\). It might be worth noting that the method presented in theorem \ref{th:weighted_poincare_1d} actually works for the full range of \(\beta\), as summed up in the following proposition
\begin{prop}
  Let \(\vphi:x\mapsto 1+ x^2\). Fix a real number \(\beta>\frac{1}{2}\). Then for all smooth bounded functions \(f\),
  \begin{equation}\label{eq:poincare_1d_sharp}
    \int f^2d\mu_{\vphi,\beta}-\pare{\int fd\mu_{\vphi,\beta}}^2\leq \frac{1}{C_\beta} \int (f')^2\vphi d\mu_{\vphi,\beta},
  \end{equation}
  where
  \begin{equation}\label{eq:Cbeta}
    C_\beta=\left\{\begin{aligned}
    2(\beta-1) & \text{if }\beta \geq \frac{3}{2}, \\
    \pare{\beta-\frac{1}{2}}^2 & \text{if }\beta \in \pare{\frac{1}{2},\frac{3}{2}}.
    \end{aligned}\right.
  \end{equation}
\end{prop}
\begin{proof}
  The proof builds on the proof for theorem \ref{th:weighted_poincare_1d}, using the explicit form of \(\vphi\). The condition \eqref{eq:CD_condition_1d} becomes
  \[
  -(2\beta-1)\pare{\frac{1}{n-1}(2\beta-1)+1}\frac{x^2}{1+x^2}+(2\beta-1)-\rho \geq 0,
  \]
  for all \(x\in\R\), or, equivalently,
  \[
  \rho\leq 2\beta-1\quad\text{and}\quad -\frac{(2\beta-1)^2}{n-1}-\rho\geq 0.
  \]
  Restricting our study to nonnegative curvatures, we thus prove that the condition \(CD(\rho,n)\) (with \(\rho>0\)) is satisfied if, and only if,
  \[
  0< \rho \leq 2\beta-1\quad\text{and}\quad 0<\rho(1-n)\leq (2\beta-1)^2\quad(\text{and }n\not\in\sbra{0,1}).
  \]
  We are looking, for a fixed \(\beta>1/2\), for the parameters \((\rho,n)\) satisfying those criteria that lead to the best possible value of \(\rho n/(n-1)\). 
  
  Assuming \((\rho,n)\) is such a couple, since \(n/(n-1)\) is decreasing on \(\R_-^*\), it is necessary that \(\rho(1-n)=(2\beta-1)^2\), which then implies that \(\rho<2(\beta-1)^2\). We thus reduce the problem to finding the maximum of the function
  \[
  \frac{\rho n}{n-1}=\rho-\frac{\rho^2}{(2\beta-1)^2}
  \]
  under the constraints \(0<\rho<\max(2\beta-1,(2\beta-1)^2)\). An easy study yields that this maximum is \(C_\beta\) as defined in equation \eqref{eq:Cbeta}. In other terms, the operator \(Lf=(1+x^2)f''+(1-\beta)2xf'\) satisfies 
  \begin{itemize}
  \item \(CD(2\beta-1,2(1-\beta))\) when \(\beta\geq 3/2\),
  \item \(CD((2\beta-1)^2/2,-1)\) when \(\beta\in\pare{1/2,3/2}\),
  \end{itemize}
  and this leads to the proposition.
\end{proof}

\begin{rmrk}
  
  This method is applicable to other functions than just \(x\mapsto 1+x^2\). For instance, for the function \(\vphi:x\mapsto 1+x^2+x^4\), one finds that the condition \(CD(2\beta-1,4\beta-1)\) is satisfied, and it is the one that leads to the best possible \(\rho n/(n-1)\) constant for the corresponding operator.
\end{rmrk}


\bibliographystyle{alpha}
\bibliography{biblio-beckner}

\end{document}